\documentclass[a4,reqno]{amsart}
\usepackage{amsmath}
\usepackage{amssymb}
\usepackage{amsthm}

\usepackage{nccmath}

\usepackage[usenames]{color}
\usepackage{eepic,epic}

\newtheorem{thm}{Theorem}[section]

\newtheorem{lem}[thm]{Lemma}

\theoremstyle{definition}
\newtheorem{defn}[thm]{Definition}
\theoremstyle{remark}
\newtheorem{rem}[thm]{Remark}
\numberwithin{equation}{section}

\begin{document}

\title[Optimal bounds for Oscillating convolution operators on the Heisenberg group]
{Optimal conditions for  $L^2$ boundedness of strongly singular convolution operators on the Heisenberg group}
\author{Woocheol Choi}
\subjclass[2000]{Primary}

\address{School of Mathematical Sciences, Seoul National University, Seoul 151-747, Korea}
\email{chwc1987@snu.ac.kr}
\maketitle
\begin{abstract} Strongly singular convolution operators $T_{K_{\alpha,\beta}}$ with the kernel $K_{\alpha,\beta}$ on the Heisenberg group $\mathbb{H}^{n}_a$ were introduced in \cite{lyall}. For case ${a^2} < C_{\beta}$, Laghi and Lyall \cite{laghi} obtained the sharp range for $\alpha,\beta$ for which the operator $T_{K_{\alpha,\beta}}$ are bounded  on $L^2 (\mathbb{H}^n_a)$. They used the classical $L^2 - L^2$ boundedness theorem for oscillatory integral operators with non-degenerate phases. But, if $a^2 \geq C_{\beta}$, the phase function related with the operators are no more non-degenerate. However,  in this paper, we obtain the sharp range for $\alpha, \beta$ for the case $a^2 \geq C_{\beta}$.  To carry out this case, we show that the canonical graph related with the phases satisfy folding type conditions and utilize the recent developed theory for degenerate oscillatory integral operators (see \cite{GR}, \cite{PS}).
\end{abstract}

\section{Introduction}
Our setting is on the Heisenberg group $\mathbb{H}^n_a$ $( a \in \mathbb{R}^{*})$ which have base manifold $\mathbb{R}^{2n+1}$ with the group law
\begin{eqnarray*}
(x,t) \cdot (y,s) = (x+ y, s+t -2 a x^{t} J y)
\end{eqnarray*}
where $J$ is the $2n \times 2n $ matrix
\begin{eqnarray*}
\begin{pmatrix} 0 & I_n \\ - I_n & 0
\end{pmatrix}
\end{eqnarray*}
($I_n$ is the $n \times n $ identity matrix). 
 This group has the following dilation law 
\begin{eqnarray*}
\lambda \cdot (x,t) = (\lambda x , \lambda^2 t ).
\end{eqnarray*}
For each kernel $K$, associated convolution operators are defined by
\begin{eqnarray*}
T_K f (x,t)  : = K * f (x,t) = \int_{\mathbb{H}^n} K\left( (x,t) \cdot (y,s)^{-1} \right) f(y,s) dy dx.
\end{eqnarray*}
We say that the operator $T_K$ is bounded on $L^{p}(\mathbb{H}^n)$ if there exist a $C > 0$ such that
\begin{eqnarray*}
\| T_K f \|_p \leq C \|f \|_p, \quad \textrm{for all } ~ f \in C^{\infty}_{0} ( \mathbb{H}^n) 
\end{eqnarray*}
And a natural quasi-norm on the Heisenberg group is defined by $\rho (x,t) = (|x|^4 + t^2 )^{1/4}$. This quasi-norm satisfies $\rho (\lambda \cdot(x,t)) = \lambda \rho (x,t)$.
For this quasi-norm, we define the strongly singular kernels for each $\alpha >0$ 
\begin{itemize}
\item $K_{\alpha, \beta} (x,t) ={\rho(x,t)^{- (2n+2+\alpha)}}e^{i\rho(x,t)^{-\beta}} \chi ( \rho(x,t)), \quad \beta>0,$
\end{itemize}
where $\chi$ is a smootho bump function in a small neighborhood of the origin. We let the operator $T_{K_{\alpha,\beta}}$ be the convolution operators on the Heisenberg group $\mathbb{H}_a^n$ with the kernel $K_{\alpha,\beta}$. This operator was firstly introduced in \cite{lyall} and there, the necessary condition is referred as  $\alpha \leq (n+\frac{1}{2}) \beta$ and it was shown that $T_{K_{\alpha,\beta}}$ is bounded if $\alpha \leq n \beta$. To show this, the group fourier transform was utilized and a lengthy calculation for estimating some oscillatory integrals was needed to obtain the results. But, Laghi and Lyall \cite{laghi} showed that we can get sharp results in the restricted case $a^2 < C_{\beta}$ for some $C_{\beta} >0$ only using the H\"ormander's $L^2-L^2$ boundedness theorems for non-degenerate oscillatory integral operators \cite{ho}. In this paper, we consider the cases $a^2 \geq C_{\beta}$ and obtain sharp conditions using the recent theory for oscillatory integral operators with degenerate phases (See section 2 for the details). The theory for the degenerate oscillatory integral operators have been developed deeply, but the use of the theorem has been restricted to the X-ray transform. 
\

Strongly singular convolution operators were studied originally in $\mathbb{R}^n$. Such operators represent some oscillating multipliers and operators of this type were first studied, using Fourier transform techniques, in the Euclidean with $\rho(x) = |x|$ by Hirschman \cite{hi} in the case d = 1 and then in higher
dimensions by Wainger \cite{Wainger}, Fefferman \cite{Fe}, and Fefferman and Stein \cite{Fe2}. 
\

On the other hand, a similar kind of convolution operators with the kernel $ \frac{1}{|x|^{n-\alpha}} e^{i |x|^{\beta}}$ with $\alpha, \beta >0$ also have been studied in \cite{PSam},\cite{Sj}. This kernel has no sinularity near the zero, but it has relatively small decaying property at infinity. Note that the case $\beta=1$ corresponds to the kernel of Bochner-Riesz means. For $\beta \neq 1$, the $L^p-L^q$ estimates including the hardy space estimates were well understood. The difference between two cases comes from the fact that the phase kernel $|x-y|^{\beta}$ is degenerate only if $\beta =1$.  We now consider an analogue problem on the Heisenberg groups with the following kernels
\begin{itemize}
\item $L_{\alpha, \beta}(x,t) = {\rho(x,t)^{- (2n+2-\alpha)}}e^{i\rho(x,t)^{\beta}} \chi (\rho(x,t)^{-1}), \quad \beta > 0.$
\end{itemize}
We denote the gruop convolution operators $T_{L_{\alpha,\beta}}$ with the kernel $L_{\alpha,\beta}$. In literatures, operators like $T_{K_{\alpha,\beta}}$ are called as strongly singular operators and $T_{L_{\alpha,\beta}}$ as oscillating convolution operators.
\

In this paper, we shall find the optimal ranges of $\alpha$ and $\beta$ where the convolution operators associated with $K_{\alpha,\beta}$ ,$L_{\alpha, \beta}$ are bounded on $L^2 (\mathbb{H}^n)$.

\
For $\frac{a^2}{b^2} \geq C_{\beta}$, the phases are no more non-degenerate. So, we need to deal with oscillatory integral operators with degenerate phases. Theory for this kind of operators have been studied largely with considering various conditions on phase functions to give a different decaying properties, See \cite{GR2}. We will  use results in  \cite{GR},\cite{PS}. To utilize such theory, we should carefully investigate the folding type for our phases. Interestingly, we have different folding types according to the values $a, b $ and $\beta$. Before stating our results, we recall the  former results in \cite{laghi}, \cite{lyall}.

\newtheorem*{zl}{Theorems \cite{laghi},\cite{lyall}} 
\begin{zl}
\
\begin{enumerate}
\item  $T_{K_{\alpha,\beta}}$ is bounded on $L^2  (\mathbb{H}^n)$ if $ \alpha \leq n \beta$.
\item  If $ 0 < {a^2} < C_{\beta}$, then $T_{K_{\alpha,\beta}}$ is bounded on $L^2 (\mathbb{H}^n)$ if and only if $\alpha \leq (n+ 1/2) \beta$. ($C_{\beta} = \frac{\beta+2}{2} ( 2\beta + 5 + \sqrt{(2\beta+5)^2 -9})$)
\end{enumerate}
\label{zl}
\end{zl}
We obtain sharp results on the $L^2 \rightarrow L^2$ boundedness of $T_{K_{\alpha,\beta}}$ for the cases $a^2 \geq C_{\beta}$. Namely, we obtain the following results:
\begin{thm}\label{main1} If ${a^2} > C_{\beta}$, then $T_{K_{\alpha,\beta}}$ is bounded on $L^2 (\mathbb{H}^n_a)$ if and only if $\alpha \leq (n+ \frac{1}{3})\beta$. If ${a^2} = C_{\beta}$, then $T_{K_{\alpha,\beta}}$ is bounded on $L^{2}(\mathbb{H}^n_a)$ if and only if $\alpha \leq (n+\frac{1}{4})\beta$.
\end{thm}
For the operators $T_{L_{\alpha,\beta}}$, we also have the sharp ranges for the $L^2 \rightarrow L^2$ boundedness except the cases $\beta = -1$ or $-2$.
\begin{thm}\label{main2} $T_{L_{\alpha,\beta}}$ is bounded on $L^2$ if and only if
\begin{enumerate}
\item[(i)] $ 0 < \beta < 1$ : For  ${a^2} < C_{\beta}$, $\alpha \leq (n+\frac{1}{2}) \beta$, for ${a^2} = C_{\beta}$, $\alpha \leq ( n+ \frac{1}{4})\beta$, and for ${a^2} > C_{\beta}$, $\alpha \leq (n+ \frac{1}{3}) \beta$.
\item[(ii)] $ 1 < \beta < 2$ : $\alpha \leq (n+\frac{1}{3}) \beta$.
\item[(iii)] $ 2 < \beta  $ : $\alpha \leq (n+ \frac{1}{2}) \beta$.
\end{enumerate}
\end{thm}

\begin{rem}
For the cases $\beta = -1$ or $\beta=-2$, we can also obtain the  sharp results for some values $\alpha,b$ where the phase becomes non-degenerate or of folding type 2. But, in these cases, higher order folding type conditions than $3$ appear for some $\alpha,b$ and the present theory for degenerate oscillatory integral estimates does not cover these cases. The theory  have been established optimally only for the one or  two folding cases (\cite{GR}, \cite{PS}).  
\end{rem}

\begin{rem} 
In $\mathbb{R}^n$, the oscillating kernel is of the form $|x|^{-\gamma} e^{i|x|^{\beta}}$ with $ \beta \neq 0$. In this case, the different behavior for the phases $|x|^{\beta}$  according to the value $a$ is characterized only by the two cases where $\beta \neq 1$ or $\beta = 1$. Precisely, we have $\det \left( \frac{\partial^2}{\partial x \partial y} |x-y|^{\beta} \right) \neq 0 $ for $x\neq y$, but  $\det \left( \frac{\partial^2}{\partial x \partial y} |x-y| \right) = 0$ for every $x \neq y$ and this is the hardest case, which is correspond to the Bochner-Riesz means operators. We see that for our cases, $\beta=1$ and $\beta = 2$ present the lowest degeneracy oscillatory operators (see \ref{det}) and it also hard to establish from the lack of the theory of oscillatory integral estimates for lower folding type cases. 
\end{rem}
\
This paper is organized as follows. In section 2, we proceed an usual process to decompose the kernel dyadically and reduce our problem to a local oscillatory integral estimates. In section 3, we recall some essential things for our oscillatory integral operators with degenerate phase functions. Finally, in section 4, we study some geometry of the canonical relation and projection maps related with the phase functions. Then, we conclude the proof of the main theorems.

\

\section{Dyadic decompostion and Localization}

We use a standard argument to reduce our problems to some oscillatory integral estimates on $\mathbb{R}^{2n+1}$. For euclidean space, the argument is well explained in \cite{St}. Though the same mechanism is easily adapted, we shall prove it because our setting is on a group. 
\

We split the kernel $K_{\alpha, \beta}$ and $L_{\alpha, \beta}$ as 
\begin{eqnarray*}
K_{\alpha, \beta} (x,t) &=& \sum_{j=0}^{\infty} \eta (2^j \rho(x,t)) K_{\alpha, \beta} (x,t)
\\
L_{\alpha, \beta} (x,t) &=& \sum_{j=0}^{\infty} \eta (2^{-j} \rho(x,t)) L_{\alpha, \beta} (x,t)
\end{eqnarray*}
where $\eta \in C^{\infty}_{0}(\mathbb{R})$ is a bump function supported in $[1/2,2]$. and satisfies $\sum_{j=0}^{\infty} \eta(2^j r) =1$ for all $ 0 < r \leq 1$. 
We let $K^{j}_{\alpha, \beta} = \eta(2^{j}(\rho(x,t))) K_{\alpha, \beta}(x,t)$ and $T_j f = K_j * f$ and $L^{j}_{\alpha, \beta} = \eta(2^{-j} \rho(x,t)) L_{\alpha, \beta} (x,t)$. For notational convenience, we omit the index $\alpha$ and $\beta$ from now.
\\
We let $T_j f = K_{\alpha, \beta}^{j} * f$ and $S_j f = L_{\alpha,\beta}^{j} * f$. Then, 
\begin{lem}
For each $N \in \mathbb{N}$, there exists a constant $C_N$ such that
\begin{eqnarray*}
\| T_j^{*} T_{j'} \|_{L^2 \rightarrow L^2} + \| T_j T_{j'}^{*}\|_{L^2 \rightarrow L^2} \leq C_N 2^{-max\{j, j'\}N}
\\
\| S_j^{*} S_{j'} \|_{L^2 \rightarrow L^2} + \| S_j S_{j'}^{*}\|_{L^2 \rightarrow L^2} \leq C_N 2^{-max\{j, j'\}N}
\end{eqnarray*}
holds when $|j - j'|  \geq c_{\beta}$ for some sufficiently large constant $c_{\beta} >0$.
\end{lem}
\begin{proof} The proof follows using the integration parts technique in usual ways.  See \cite{lyall} where the proof for $T_j$ is given.
\end{proof}
 By the Cotlar-Stein Lemma, we only need to show that 
\begin{eqnarray*}
\| T_j \|_{L^2 \rightarrow L^2} + \| S_j \|_{L^2 \rightarrow L^2} \leq C 
\end{eqnarray*}
holds uniformly for $j \in \mathbb{N}$ with some constant $C>0$.

\

Let
\begin{eqnarray*}
\tilde{K}^{j}_{\alpha,\beta} (x,t) &=& K_{\alpha,\beta}^{j}(2^{-j} \cdot (x,t)) = \eta(\rho(x,t)) 2^{j(Q+\alpha)} \rho(x,t)^{-Q-\alpha} e^{i2^{j\beta} \rho(x,t)^{-\beta}},
\\
\tilde{L}^{j}_{\alpha,\beta} (x,t) &=& L_{\alpha,\beta}^{j}(2^{-j} \cdot (x,t)) = \eta(\rho(x,t)) 2^{-j(Q-\alpha)} \rho(x,t)^{-Q+\alpha} e^{i2^{j\beta} \rho(x,t)^{\beta}}.
\end{eqnarray*}
Let $f_j (x,t) = f(2^{j} \cdot (x,t))$. Then $K^{j}*f (2^{-j}\cdot (x,t)) = 2^{-jQ} (\tilde{K}^{j} * f_j ) (x,t)$.

\begin{eqnarray*}
\|K^{j}_{\alpha,\beta} * f (x,t) \|_{L^2} & = & 2^{-jQ/2} \| K_{\alpha,\beta} * f (2^{-j} \cdot (x,t))\|_{L^2}
\\
&\leq & 2^{-jQ/2}\cdot 2^{-jQ} \|\tilde{K}_{\alpha,\beta}^{j} * f(2^j \cdot()) (x,t) \|_{L^2}
\\
&\leq& 2^{-jQ/2} \cdot 2^{-jQ} \|\tilde{K}_{\alpha,\beta}^{j} \|_{L^2 \rightarrow L^2} \| f (2^{-j} \cdot ) \|_{L^2}
\\
&=& 2^{-jQ} \|\tilde{K}_{\alpha,\beta}^{j} \|_{L^2 \rightarrow L^2} \|f\|_{L^2}.
\end{eqnarray*}
Similarly, we also have $\| L_{\alpha,\beta}^{j} * f \|_{L^2} \leq 2^{jQ} \| \tilde{L}_{\alpha,\beta}^{j} \|_{L^2 \rightarrow L^2} \| f \|_{L^2}$.

\

So, it suffice to prove suitable boundedness for $\tilde{T}_j$ and $\tilde{S}_j$. Now, we further modify our operator locally using the fact that the kernels of $\tilde{T}_j$ and $\tilde{S}_j$ are supported in $\{ (x,t) : \rho(x,t) \leq 2 \}$. To do this, we find a set of point $ G = \{ g_k : k \in \mathbb{N} \}$ such that $\bigcup_{k\in \mathbb{N}} B(g_k, 2 ) = \mathbb{H}^{n}_a $ and each $B( g_k, 4)$ contains only $d_n$'s other $g_l$ members in $G$.

\

We can split $f = \sum_{k=1}^{\infty} f_k$ with each $f_k$ supported in ${B(g_k ,2 )}$. And we let 
\begin{eqnarray*}
\tilde{T}_j^{l,k} f (x,t) = \int \tilde{K}_{\alpha,\beta}^{j} \left( (x,t) \cdot (y,s)^{-1} \right) \cdot \eta \left( \rho\left( (x,t) \cdot g_k^{-1}\right) \right) 
\eta \left( \rho\left( (x,t) \cdot g_l^{-1}\right) \right)  f(y,s) dy ds
\end{eqnarray*}

Then, 
\begin{eqnarray*}
\| \tilde{T}_j * f \|^{2}_{L^2(\mathbb{H}_a^{n})} &\leq& \sum_{k=1}^{\infty} \| \tilde{T}_j * f \|^2_{L^2(B(g_k,2))} 
\\
&\leq& \sum_{k=1}^{\infty} \| \tilde{T}_j * \sum_{l=1}^{\infty} f_l \|_{L^2 ( B(g_k,2))}^2
\\
&=& \sum_{k=1}^{\infty} \| \tilde{T}_j * \sum_{\{ l : \rho(g_l \cdot g_k^{-1}) \leq 2 \}} f_l \|_{L^2((B(g_k,2)))}^{2}
\\
&\lesssim& \sum_{k=1}^{\infty} \sum_{l : \rho(g_l \cdot g_k^{-1}) \leq 2} \|\tilde{T}_{j}^{l,k} \|_{L^2 \rightarrow L^2} \| f_l \|_{L^2}^2
\\
&\lesssim&  \sup_{\substack{\rho ( g_l \cdot g_k^{-1}) \leq 2}} \|\tilde{T}_j^{l,k} \|_{L^2 \rightarrow L^2}  \|f \|_{L^2}^{2}.
\end{eqnarray*} 
Because $ \det \left(D_{x,t} \left( (x,t) \cdot g\right) \right) = 1$ for all $g \in \mathbb{H}^{n}_a$, 
\begin{eqnarray*}
\tilde{T}_j^{l,k} f \left( (x,t)\cdot g_k\right) = \int \tilde{K}^{j} \left( (x,t)\cdot (y,s)^{-1}\right) \eta(\rho(x,t))\eta(\rho(y,s)\cdot (g_k\cdot g_l^{-1})) f((y,s)\cdot g_k) dy ds.
\end{eqnarray*}
Notice that $\rho(g_k\cdot g_l^{-1}) \lesssim 1 $. If we let $\psi \left( (x,t),(y,s)\right)= \eta (\rho(x,t)) \eta(\rho((y,s)\cdot (g_k\cdot g_l^{-1})))$ and substitute $f$ as $\tilde{f}(x) := f (x\cdot g_k)$.

$\sum_{j,k , \rho(g_l \cdot g_k^{-1}) \leq 2} \| \tilde{T}_j^{l,k}\|$  will be obtained if we prove $\| A_j \|_{L^2 \rightarrow L^2} \lesssim 1$ for 
\begin{eqnarray*}
A_j f (x,t) = \int \tilde{K}_{\alpha,\beta}^{j} \left( (x,t)\cdot(y,s)^{-1} \right) \psi \left( (x,t) , (y,s) \right) f (y,s) dy ds
\end{eqnarray*}
with a compactly supported smooth function $\psi$. 
Since
\begin{eqnarray*}
\tilde{K}^{j}_{\alpha,\beta} (x,t) &=& \eta(\rho(x,t)) 2^{j(Q+\alpha)} \rho(x,t)^{-Q - \alpha} e^{i 2^{j\beta} \rho(x,t)^{-\beta}}
\\
\tilde{L}^{j}_{\alpha,\beta} (x,t) &=& \eta(\rho(x,t)) 2^{- j(Q+\alpha)} \rho(x,t)^{-Q + \alpha} e^{i 2^{j\beta} \rho(x,t)^{\beta}}
\end{eqnarray*}
The matters are reduced to show that $\|T_{{A}_j}  \|_{L^2 \rightarrow L^2} \lesssim 2^{jQ}$ and $ \|T_{B_j}  \|_{L^2 \rightarrow L^2} \lesssim 2^{-jQ}$ 
Let 
\begin{eqnarray*}
A_j (x,t) &=& 2^{j\alpha} \mu(x,t) e^{i 2^{j\beta} \rho(x,t)^{-\beta}}
\\
B_j (x,t) &=& 2^{- j\alpha} \mu (x,t) e^{i 2^{j\beta} \rho(x,t)^{\beta}}
\end{eqnarray*}
where with $\mu$ is a smooth funtion supported on $\rho(x,t) \sim 1$.
We define the operators $L_{A_j}$, $L_{B_j}$ by
\begin{eqnarray*}
L_{A_j} f (x,t) = \int A_j \left( (x,t) \cdot (y,s)^{-1}\right) \psi\left( (x,t), (y,s)\right) f(y,s) dy ds
\\
L_{B_j} f(x,t) = \int B_j \left( (x,t) \cdot (y,s)^{-1} \right) \psi\left( (x,t),(y,s)\right) f(y,s) dy ds.
\end{eqnarray*}
To prove Theorem \ref{main1} and Theorem \ref{main2}, it suffices to establish the following two theorems.
\begin{thm}\label{LA} The following inequalities hold uniformly for $j$.
\

If $a^2 > C_{\beta}$, \quad $\| L_{A_j}\|_{L^2 \rightarrow L^2} \lesssim 2^{j (\alpha - (n+\frac{1}{3})\beta)}$,
\

If $a^2 = C_{\beta}$, \quad $\|L_{A_j}\|_{L^2 \rightarrow L^2} \lesssim 2^{j (\alpha - (n+\frac{1}{4})\beta)}$.
\end{thm}

\begin{thm}\label{LB} The following inequalities hold uniformly for $j$.
\\
$(i) -1 < \beta < 0, $
\

If $a^2 < C_{\beta}$, \quad$\|L_{B_j}\|_{L^2 \rightarrow L^2} \lesssim 2^{j (\alpha -(n+\frac{1}{2})\beta)}$,
\

If $a^2 = C_{\beta}$, \quad$\|L_{B_j}\|_{L^2 \rightarrow L^2} \lesssim 2^{j (\alpha - (n+\frac{1}{4})\beta)}$.
\\
$(ii) -2 < \beta < -1,$ \quad$\|L_{B_j}\|_{L^2 \rightarrow L^2} \lesssim 2^{j (\alpha - (n+\frac{1}{3})\beta)}$.
\\
$(iii)  ~\beta < -2,$\qquad $\| L_{B_j} \|_{L^2 \rightarrow L^2} \lesssim 2^{j (\alpha -(n+\frac{1}{2})\beta)}$.
\end{thm}
In the next section, we shall breifly review on the theory related to the operators $L_{A_j}$ and $L_{B_j}$. And in section 4, we will expoit some required geometry related with the phase function $\rho(x,t)^{\beta}$ to conclude the proof of Theorem \ref{LA} and Theorem \ref{LB}.
\section{$L^2$ theory for oscillatory integral operators}
In this section, we review on some $L^2 \rightarrow L^2 $ theory for oscillatory integral operators. The form of operators we are concern is given as
\begin{eqnarray*}
T^{\phi}_{\lambda} f (x) = \int_{\mathbb{R}^n} e^{i \lambda \phi(x,y)} a(x,y) f (y) dy
\end{eqnarray*}
where $\phi, a \in C^{\infty}(\mathbb{R}^n \times \mathbb{R}^n)$ and $a$ has a compact support. We firstly state the fundamental theorem of H\"ormander \cite{ho}.
\begin{thm} Suppose that the phase function $\phi$ satisfies $\det\left( \frac{\partial^2 \phi}{\partial x_i \partial y_j} \right) \neq 0$ on the support of $a$. Then we have the following inequality
\begin{eqnarray*}
\|T^{\phi}_{\lambda} \|_{L^2 \rightarrow L^2} \lesssim \lambda^{-\frac{n}{2}},\quad \lambda \geq 1
\end{eqnarray*}
where the implicit constant is independent of $\lambda$. 
\end{thm}
We say that $\phi$ is non-degenerate if it satisfies the assumption of the above theorem. And we say $\phi$ is degenerate if there is some point $(x_0, y_0)$ where $\det \left.\left(\frac{\partial^2}{\partial x_i \partial y_j}\right)\right |_{(x_0,y_0)}$ equals to zero.
This theorem gives the sharp decaying of $\|T_{\lambda}^{\phi}\|_{L^2\rightarrow L^2}$ in terms of $\lambda$. But, the phase functions of our operators (see  ) can become degenerate according to the values of $a$ and $\beta$. For a degenerate $\phi$, the optimal number $\kappa_{\phi}$ for which the inequality $\|T_{\lambda}\|_{L^2 \rightarrow L^2} \lesssim \lambda^{-\kappa_{\phi}}$ holds would be less than $\frac{n}{2}$. The number $\kappa_{\phi}$'s are related to the folding type conditions of the phase $\phi$ (See below). For phases whose folding degrees are $\leq 3$, the sharp numbers $\kappa_{\phi}$ were obtained in \cite{GR}, \cite{PS}. We shall use the results. In fact, the results for folding types $\leq 3$ in \cite{GR} is the best known results and there are no results for folding types $>3$ except the very restricted result in \cite{cu}. Finding completely the numbers $\kappa_{\phi}$ corresponding to all degenerate phases $\phi$ seems a widely open problem.

It is well-known that the decaying property is strongly related th the geometry of the canonical relation
\begin{eqnarray}\label{relation1}
C_{\phi} = \{ (x, \phi_x (x,y), y, - \phi_y (x,y) ) \subset T^{*} (\mathbb{R}^n_x) \times T^{*} (\mathbb{R}^n_y)
\end{eqnarray}
To describe the geometry of $C_{\phi}$, we need the following definition

\begin{defn}\label{def}
Let $M_1$,$M_2$ be smooth manifolds of dimension $n$, and $f : M_1 \rightarrow M_2$ is a smooth map of corank $\leq 1$. We let the singular set $S = \{ P \in M_1 : \textrm{rank} (Df)  < n ~\textrm{at}~ P\}.$ Then we say that $f$ has a $k-$ type fold at $P_0$ for $P_0 \in S$ if
\begin{enumerate}
\item $\textrm{rank}(Df)|_{P_0} = n-1$,
\item $\det (Df)$ vanishes of $k$ order in the null direction at $P_0$.
\end{enumerate}
Here, the null direction means the unique direction vector $v$ such that $(D_v f)|_{P_0} =0$.
\end{defn}
\
Now we consider the two projection maps
\begin{eqnarray}\label{relation2}
\pi_L : C_{\Phi} \rightarrow T^{*} ( \mathbb{R}^n_x) \quad \textrm{and} \quad \pi_R : C_{\Phi} \rightarrow T^{*} (\mathbb{R}^n_y).
\end{eqnarray}
Then the following theorem is proved in \cite{PS} for one folds cases and in \cite{GR} for two folds cases.
\begin{thm}\label{degenerate}
Suppose that the projection maps $\pi_L$ and $\pi_R$ have one fold singularities, then 
\begin{eqnarray*}
\| T_{\lambda} f \|_{L^2(\mathbb{R}^n)} \leq C \lambda^{-\frac{(n-1)}{2}-\frac{1}{3}} \| f \|_{L^2(\mathbb{R}^n)}.
\end{eqnarray*}
If the projection maps $\phi_L$ and $\pi_R$ have two hold singularities, 
\begin{eqnarray*}
\| T_{\lambda} f \|_{L^2(\mathbb{R}^n)} \leq C \lambda^{-\frac{(n-1)}{2}-\frac{1}{4}} \| f \|_{L^2(\mathbb{R}^n)}.
\end{eqnarray*} 
\end{thm}

\

\section{Geometry of the Canonical relation maps}
In this section, we study the projection maps \eqref{relation2} associated with the phase function of our operators in Theorem \ref{LA} and Theorem \ref{LB} to use Theorem \ref{degenerate}. Recall that $\rho(x,t) = (|x|^4 + t^2)^{1/4}$ and the phase function $\phi$ of the integral operators $L_{A_j}$ and $L_{B_j}$  is 
\begin{eqnarray*}
\phi(x,t,y,s) = \rho^{-\beta} \left( (x,t)\cdot (y,s)^{-1}\right).
\end{eqnarray*}
To write the group law explicitly, we write $x = (x^1, x^2)$ and $y = (y^1, y^2)$ with $x^1, x^2, y^1, y^2 \in \mathbb{R}^n$. Let $\Phi(x,t) = \rho(x,t)^{-\beta}$ for simplicity, then we have
\begin{eqnarray}\label{phi}
\phi(x, t, y, s) = \Phi \left( x^1 - y^1, x^2 - y^2, t - s - 2a (x^1 y^2 -x^2 y^1)\right)
\end{eqnarray}
Let us use the derivatives $\partial_x, \partial_t$ to denote the position derivatives of $\Phi$ (i,e. $\partial_{x_j} \Phi = \Phi_{x_j}, \partial_t \Phi = \Phi_t ~\textrm{and}~ \partial_{x_j} (x_j \Phi) = x_j \Phi_{x_j})$.
\

Firstly, to determine the phase function $\Phi$ is non-degenerate or not, we should calculate the determinant of the matrix
\begin{eqnarray*}
H=\left( \frac{\partial^2  \phi(x,t,y,s) }{\partial_{(y,s)} \partial_{(x,t)}} \right).
\end{eqnarray*}
So, we now compute the matrix $H$. In \eqref{phi}, by the chain-rule, we have
\begin{eqnarray*}
\frac{\partial}{\partial x_j} \phi(x,t,y,s) = (\partial_{x_j} + 2 a y_{n+j} \partial_t ) \Phi |_{\left( (x,t) \cdot (y,s)^{-1}\right)}
\\
\frac{\partial}{\partial x_{j+n}} \phi(x,t,y,s) = (\partial_{x_{j+n}} - 2 a y_j \partial_t ) \Phi|_{\left( (x,t) \cdot (y,s)^{-1}\right)}
\end{eqnarray*}
Let 
\begin{eqnarray*}
A(y) = \begin{pmatrix} I & 2a J y \\ 0 & 1 \end{pmatrix} 
\end{eqnarray*}
Using a similar calculation, we see that 
\begin{eqnarray}\label{matrix}
H(x,t,y,s) &=& A(x) \left( \partial_i \partial_j \Phi \right) A(y)^t + 2 a \partial_t \Phi \left. \begin{pmatrix} J & 0 \\ 0 &0\end{pmatrix}\right|_{\left( (x,t) \cdot (y,s)^{-1}\right)}
\\
&=& A(x) \left. \left[ (\partial_i \partial_j \Phi ) + 2 a \partial_t \Phi \begin{pmatrix} J & 0 \\ 0 &0 \end{pmatrix} \right] \right|_{\left( (x,t) \cdot (y,s)^{-1}\right)}A(y)^{t}
\end{eqnarray}
where the second equality holds because {\footnotesize$A(x)  \begin{pmatrix}J&0\\0&0\end{pmatrix} A(y)^{t} = \begin{pmatrix}J&0\\0&0\end{pmatrix}$.} We let
\begin{eqnarray}\label{L1}
L(x,t,y,s) = \left. \left[ (\partial_i \partial_j \Phi ) + 2 a \partial_t \Phi \begin{pmatrix} J & 0 \\ 0 &0 \end{pmatrix} \right] \right|_{\left( (x,t) \cdot (y,s)^{-1}\right)}.
\end{eqnarray}
Thus, to study the matrix $H$, it is enough to analyze the matrix $L$. Morerover, we have  $\det(A(x)) = \det(A(y))=1$, which implies  $ \det (H(x,t,y,s)) = \det ( L(x,t,y,s)).$ In \cite{laghi}, the determinant of the matrix $H$ was calculated directly without using this observation. Here we calculate it by obtaining the determinant of $L$.

We now calculate the hessian matrix of $\Phi$. For $1 \leq i, j \leq 2n$,

\begin{eqnarray*}
\partial_j \Phi (\textrm{x}
,\textrm{t}) & = &- \frac{\beta}{4} (|\textrm{x}
|^4 + \textrm{t}^2)^{-\frac{\beta}{4} -1} ( 4 \textrm{x}
_j |\textrm{x}
|^2)
\\
\partial_t \Phi(\textrm{x}
,\textrm{t}) & = & - \frac{\beta}{4}(|\textrm{x}
|^4 + \textrm{t}^2)^{-\frac{\beta}{4}-1}(2\textrm{t})
\end{eqnarray*}
and
\begin{eqnarray*}
\partial_i \partial_j \Phi(\textrm{x}
,\textrm{t})&=& \beta(|\textrm{x}
|^4 + \textrm{t}^2)^{-\frac{\beta}{4} -2} \left[ (\beta+4) |\textrm{x}
|^4 - 2 (|\textrm{x}
|^4+ \textrm{t}^2) \right] \textrm{x}
_i \textrm{x}
_j - \beta(|\textrm{x}
|^4 + \textrm{t}^2)^{-\frac{\beta}{4}-1} \delta_{ij} |\textrm{x}
|^2
\\
\partial_i \partial_t \Phi(\textrm{x}
,\textrm{t})&=& \beta(\beta+4) (|\textrm{x}
|^4 + t^2)^{-\frac{\beta}{4}-2} |\textrm{x}
|^2  x_i \cdot \frac{\textrm{t}}{2}
\\
\partial_t^2 \Phi(\textrm{x}
,\textrm{t}) &=& \beta (\beta+4) (|\textrm{x}
|^4 + \textrm{t}^2 )^{-\frac{\beta}{4} -2}  \frac{\textrm{t}}{2} \cdot \frac{\textrm{t}}{2} - \beta (|\textrm{x}
|^4 + \textrm{t}^2)^{-\frac{\beta}{4}-1} \frac{1}{2}.
\end{eqnarray*}
Let $D = (|\textrm{x}
|^2 \textrm{x}
, \frac{\textrm{t}}{2} )^{t}$, then the above computations show that
\begin{eqnarray}\label{L2}
&\left. (\partial_i \partial_j \Phi) + 2 a \partial_t \Phi \begin{pmatrix} J & 0 \\ 0&0 \end{pmatrix} \right|_{ (\textrm{x}
,\textrm{t}) } &
\\&= \beta(\beta + 4) (|\textrm{x}
|^4 + \textrm{t}^2 )^{-\frac{\beta}{4} -2} D \cdot D^{t} - &\beta (|\textrm{x}
|^4 + \textrm{t}^2)^{-\frac{\beta}{4} -1} \begin{pmatrix} |\textrm{x}
|^2 I + a\textrm{t} J + 2 \textrm{x}
 \cdot \textrm{x}
^{t} & 0 \\ 0 & \frac{1}{2} \end{pmatrix}
\\&=-\beta (|\textrm{x}
|^4 + \textrm{t}^2 )^{-\frac{\beta}{4}-1} ( E + R )
\end{eqnarray}
To write the matrices simply, we let $B = |\textrm{x}|^2 I + a \textrm{t}J$, $K = \textrm{x}\cdot \textrm{x}^t$ and
\begin{eqnarray*}
E =\begin{pmatrix} B+2K  & 0 \\ 0 & \frac{1}{2}\end{pmatrix} \quad\textrm{and} \quad R= - \frac{(\beta+4)}{|\textrm{x}
|^4 + \textrm{t}^2} D\cdot D^{t}.
\end{eqnarray*}
Then, in \eqref{L1} and \eqref{L2}, we have
\begin{eqnarray}\label{sowe}
L(x,t,y,s) = [- \beta(|\textrm{x}
|^4 + \textrm{t}^2 )^{-\frac{\beta}{4}-1} (E + R)]_{(\textrm{x}
,\textrm{t}) = (x,t)\cdot (y,s)^{-1}}
\end{eqnarray}
Now we are ready to obtain the following Lemma
\begin{lem}\label{det} We have
\begin{eqnarray*}
\det H (x,t,y,s) = F( (x,t) \cdot (y,s)^{-1})
\end{eqnarray*}
where $F(x,t) = c_{a,\beta} (|x|^4 + a^2 t^2)^{m_1} (|x|^4 + t^2 )^{m_2} f (x,t)$  for some $m_1,m_2, c_{a,\beta} \in \mathbb{R}$ and  $f(x,t) = 2(\beta+1)|x|^8 + (3(\beta+2) - 2a^2)|x|^4 t^2 + (\beta+2) a^2 t^4$.
\end{lem}
\begin{proof} 
From \eqref{matrix} and \eqref{sowe}, it is enought to show that
\begin{eqnarray*}
\det [-\beta(|\textrm{x}|^4+ \textrm{t}^2)^{-\frac{\beta}{4}-1} (E + R) ] = F (\textrm{x},\textrm{t}).
\end{eqnarray*}
Moreover, from the form of given $F$ in the theorem, we only need to compute $\det (E+R)$.
\
We recall the form
\begin{eqnarray*}
E + R = \begin{pmatrix} B + 2 K & 0 \\ 0 & \frac{1}{2} \end{pmatrix} - \frac{(\beta+4)}{|x|^4 + t^2} D \cdot D^{t}.
\end{eqnarray*}
\
For notational convenience, we always use case-letter $f_i$ to denote the $i'$th row of matrix capital $F$. 
Notice that $D \cdot D^{t} $ is of rank $1$ and will exploit the following convention 
\begin{eqnarray}\label{rank}
\det ( P + Q) = \det(P) + \sum_{j=1}^{m} \det \begin{pmatrix} p_1 \\ \vdots \\ p_{j-1} \\ q_j \\ p_{j+1} \\ \vdots \\ p_{m}\end{pmatrix} 
\end{eqnarray}
for any $m \times m$ matrix $P$ and $Q$ with rank $Q =1$. Recall $B = |x|^2 + at J$ and $K=x\cdot x^t$, then direct calculations show that  
\begin{eqnarray}\label{detB}
\det (B) = (|x|^4 + a^2 t^2)^n
\end{eqnarray}
and 
\begin{eqnarray}\label{detK}
\sum_{j=1}^{n} x_j \det \begin{pmatrix} b_1 \\ \vdots \\ b_{j-1} \\ k_j \\ b_{j+1} \\ \vdots \\ b_{2n}\end{pmatrix} 
 + \sum_{j=1}^{n} x_{j+n} \det \begin{pmatrix} b_1 \\ \vdots \\ b_{j+n-1} \\ k_{j+n} \\ b_{j+n+1} \\ \vdots \\ b_{2n}\end{pmatrix}  
\\
= \sum_{j=1}^{n} x_j ( |x|^2 x_j + x_{n+j} a t) (|x|^4 + a^2 t^2)^{n-1} + \sum_{j=1}^{n} x_{j+n} &(|x|^2 x_{j+n} - x_j a t ) (|x|^4 + a^2 t^2)^{n-1}
\\
= (|x|^4 + a^2 t^2)^{n-1} |x|^4
\end{eqnarray}
Thus, from \eqref{rank}, \eqref{detB} and \eqref{detK}, we have
\begin{eqnarray}\label{det2}
\det (B+2K) &=& (|x|^4 + a^2 t^2)^{n} + 2 |x|^4 (|x|^4 +a^2 t^2)^{n-1} 
\\
&=& (|x|^4+a^2 t^2)^{n-1}(3|x|^4 + a^2 t^2)
\end{eqnarray}
Using  \eqref{rank} again, we have 
\begin{eqnarray*}
\det(E+R) &=& \det(E) + \frac{1}{2} \sum_{j=1}^{2n} \det \begin{pmatrix} e_1 \\ \vdots \\ e_{j-1} \\ r_j \\ e_{j+1} \\ \vdots \\ e_{2n} \end{pmatrix} 
+ \det \begin{pmatrix} e_1 \\ \vdots \\ e_{2n} \\ r_{2n+1} \end{pmatrix}
\\
&=:& S_1 + S_2 + S_3 
\end{eqnarray*}
From  \eqref{det2}
\begin{eqnarray*}
S_1 = \det \begin{pmatrix} B + 2K & 0 \\ 0& \frac{1}{2}\end{pmatrix} = \frac{1}{2} \det ( B + 2K)  = \frac{1}{2} (|x|^4 + a^2 t^2)^{n-1} (3|x|^4 + a^2 t^2)
\end{eqnarray*}
Since $\textrm{rank} ~K =1$, we see that
\begin{eqnarray*}
\det \begin{pmatrix} e_1 \\ \vdots \\ e_{j-1} \\ r_j \\ e_{j+1} \\ \vdots \\ e_{2n} \end{pmatrix} = \det \begin{pmatrix} b_1 + 2 k_1 \\ \vdots \\ b_{j-1} + 2 k_{j-1} \\ \frac{-(\beta+4)|x|^4}{|x|^4 +t^2} k_j \\ b_{j+1} + 2 k_{j+1} \\ \vdots \\ b_{2n} + 2 k_{2n} \end{pmatrix} = - \frac{(\beta+4)|x|^4}{|x|^4 + t^2}  \det  \begin{pmatrix} b_1 \\ \vdots \\ b_{j-1} \\ k_j \\ b_{j+1} \\ \vdots \\ b_{2n} \end{pmatrix}. 
\end{eqnarray*}
So, 
\begin{eqnarray*}
S_2 &=& - \frac{1}{2} \left(\frac{(\beta+4)|x|^4}{|x|^4 + t^2} \right) |x|^4 (|x|^4 + a^2 t^2)^{n-1}
\end{eqnarray*}
Finally, 
\begin{eqnarray*}
S_3 &=& \det \begin{pmatrix} E + 2K & 0 \\ * & - \frac{\beta+4}{|x|^4 + t^2}\frac{t^2}{4} \end{pmatrix} = - \frac{\beta+4}{|x|^4+t^2} \frac{t^2}{4} \det (B+2K) 
\\
&=& -\frac{\beta+4}{|x|^4 + t^2} \frac{t^2}{4} (|x|^4 +a^2 t^2)^{n-1} (3|x|^4 +a^2 t^2)
\end{eqnarray*}
Adding all these terms, we have
\begin{eqnarray*}
\det \left( (\partial_i \partial_j \Phi) (x,t) + 2a \partial_t \Phi(x,t) \begin{pmatrix} J & 0 \\ 0 & 0 \end{pmatrix} \right) = p (|x|^4 + a^2 t^2) q (|x|^4 + t^2) f(x,t)
\end{eqnarray*}
where $p(r) = c_p r^{m_1}$, $q(r) = r^{m_2}$ for some $m_1, m_2, c_p \in \mathbb{R}$ and 
\begin{eqnarray*}
f(x,t) = 2 (\beta+1) |x|^8 + ( 3 (\beta+2) - 2a^2) |x|^4 t^2 + (\beta+2)a^2 t^4.
\end{eqnarray*} 
\end{proof}
Now, we should determine when the determinant of $H(x,t,y,s)$ can be zero for some values $(x,t,y,s)$ with $\rho\left( (x,t)\cdot (y,s)^{-1}\right) \sim 1$. Furthermore, to determin the folding type in the degenerate case, it will be decisive to know the shape of factorization. We take it in the following theorem
\begin{lem}\label{factor} For some nonzero constants $\gamma, c,c_1,c_2,c_3$ with $c_1 \neq c_2$ and $c_3 >0$ which are determined by $\beta$ and $a$, 
\

\begin{itemize}
\item case 1 : $\beta \in (-1,0) \cup (0,\infty)$
\begin{itemize}
\item[$\cdot$] For $\frac{a^2}{b^2} < C_{\beta}$, $f(x,t) > 0$.
\item[$\cdot$] For $\frac{a^2}{b^2} = C_{\beta}$, $f(x,t) = \gamma ( |x|^2 - c t^2)^2$.
\item[$\cdot$] For $\frac{a^2}{b^2} > C_{\beta}$, $f(x,t) = \gamma ( |x|^2 - c_1 t)(|x|^2 + c_1 t)(|x|^2 - c_2 t)(|x|^2 + c_2 t)$.
\end{itemize}

\item case 2 : $\beta \in (-2,-1)$
\begin{itemize}
\item[$\cdot$] $f(x,t) = \gamma (|x|^2- c_1 t) (|x|^2 + c_1 t) (|x|^4 + c_3 t^2)$.
\end{itemize}
\item case 3 : $\beta \in (\infty, -2)$
\begin{itemize}
\item[$\cdot$] $ f(x,t) < 0$.
\end{itemize}
\end{itemize}
\end{lem}
\begin{proof}

Firstly, we see that $f(x,t) >0$ for $3(\beta +2) - 2a^2 > 0$. And $f(x,t) >0 $ also holds if the dicriminant 
\begin{eqnarray*}
\Delta = 4a^{4} - 4 (\beta+2) (2\beta + 5) a^2 + 9 (\beta +2)^2
\end{eqnarray*}
is negative. This range is equal to
\begin{eqnarray*}
C_{\beta}^{-} < a^2 < C_{\beta}^{+}
\end{eqnarray*} 
where
\begin{eqnarray*}
C_{\beta}^{\pm} = \frac{\beta +2}{2} \left(2\beta +5 \pm \sqrt{(2\beta  + 5)^2 - 9 }\right) .
\end{eqnarray*}
But, 
\begin{eqnarray*}
C_{\beta}^{-} = \frac{(\beta+2)}{2} (2\beta + 5 - \sqrt{(2\beta +5)^2 -9}) & = & \frac{(\beta+2)}{2} (2\beta + 5 - \sqrt{(2\beta+2)(2\beta+8)}
\\
& < & \frac{(\beta+2)}{2} (2 \beta+5 - \sqrt{(2\beta+2)^2}) = \frac{3 (\beta+2)}{2}
\end{eqnarray*}
So, we compress two conditions as $f(x,t) >0$ for $a^2 < C_{\beta}^{+}$.

\end{proof}

For degenerate cases, we need to analyze the canonical relation \eqref{relation1} associated with our phase function $\Phi$
\begin{eqnarray*}
C_{\Phi} = \{ \left( (x,t) , \Phi_{(x,t)}, (y,s), - \Phi_{(y,s)} \right) \} \subset T^{*} (\mathbb{R}^{2n+1}) \times T^{*} (\mathbb{R}^{2n+1})
\end{eqnarray*}
and the projection maps \eqref{relation2}
\begin{eqnarray*}
\pi_L : C_{\Phi} \rightarrow T^{*} (\mathbb{R}^{2n+1}) \quad \textrm{and} \quad \pi_R : C_{\Phi} \rightarrow T^{*}(\mathbb{R}^{2n+1}).
\end{eqnarray*}
We will check the condition (1) and (2) of Definition \ref{def} to prove the following theorem
\begin{thm}\label{folding} On the hypersurface $S$, 
\begin{enumerate}
\item If $\beta \in (-2,-1)$ or  $\beta \in (-1,0) \cup (0,\infty)$ and $\frac{a^2}{b^2} > C_{\beta}$, the projection maps $\pi_L$ and $\pi_R$ are both of folding type 1. 
\item If $\beta \in (-1,0) \cup (0, \infty)$ and $\frac{a^2}{b^2} = C_{\beta}$, the projection maps $\pi_L$ and $\pi_R$ are both of folding type 2.
\end{enumerate}
\end{thm} 
Let
\begin{eqnarray*}
S = \{ (x,t,y,s) : \det H (x,t,y,s) =0 \}.
\end{eqnarray*}
We need the following lemma to show that rank of $\pi_L$ and $\pi_R$ drop by  1 simply.

\begin{lem}\label{rank} Let $L_1(x,t,y,s)$ be the first $(2n) \times (2n)$matrix of $L(x,t,y,s)$ and suppose that $\left( x,t,y,s\right)$ is contained in $S$. Then, 
\begin{eqnarray*}
\det L_1 (x,t,y,s) \neq 0
\end{eqnarray*}
provided $\beta \neq -4$.
\end{lem}

\begin{proof}
For simplicity, let $(z,w) := (x,t)\cdot (y,s)^{-1}$. Except the nonzero common facts,  we only need to check the determinant of 
\begin{eqnarray*}
M(z,w)=  \begin{pmatrix} |z|^2 I + a w J + 2 z \cdot z^t  - ( \beta+4) \frac{|z|^4}{|z|^4 + w^2 } x \cdot z^t \end{pmatrix}
\end{eqnarray*}
is nonzero for $(z,w) \neq (0,0)$. It can be calculated in the same way using \eqref{det1} and \eqref{det2} so that
\begin{eqnarray*}
\det(M(z,w)) = (|z|^4 + a^2 w^2)^n + (|z|^4 + a^2 w^2 )^{n-1} |z|^4 \left( 2 - (\beta+4) \frac{|z|^4}{|z|^4 + w^2} \right)
\\
= \frac{(|z|^4+a^2 w^2)^{n-1}}{|z|^4 + w^2} \left[  - (\beta+1) |z|^8 + (a^2 + 3) |z|^4 w^2 + a^2 w^4 \right] 
\end{eqnarray*}
But, $(z,w)$ is on $S$ and satisfies
\begin{eqnarray}\label{zero}
2(\beta+1)|z|^8 + (3 (\beta+2) - 2a^2) |z|^4 w^2 + (\beta+2) a^2 w^4 = 0 
\end{eqnarray} 
From above two equalites, we have
\begin{eqnarray*}
 (3(\beta+2) + 6) |z|^4 w^2 + (\beta +4) a^2 w^4 = 0 
\end{eqnarray*}
So
\begin{eqnarray*}
\det(M(z,w)) =  \frac{(|z|^4+a^2 w^2)^{n-1}}{|z|^4 + w^2} \frac{w^2}{2}  (\beta+4)\left[ 3|z|^4 + a^2 w^2 \right]
\end{eqnarray*}
If $w=0$, then $z$ becomes zero in \eqref{zero}. Because $(z,w) \neq (0,0)$, $w$ should be nonzero. Thus $\det (M(z,w)) \neq 0$.

\end{proof}

\begin{proof}[proof of Theorem \ref{folding}]
We will only consider the case (1) in the theorem. The other case can be derived similrary. It is only remained to show that at the hypersurface $S$, we have the folding type conditions as in the theorem. Recall that
\begin{eqnarray*}
S &= &\{ (x,t,y,s) \in \mathbb{R}^{(2(2n+1)} ~|~ F\left((x,t)\cdot (y,s)^{-1}\right) = F\left( x-y, s-t + 2a x^t J y\right) =0 , \quad \rho\left( (x,t)\cdot (y,s)^{-1}\right) \sim 1 \}
\\
&=& \{ (x,t,y,s) \in \mathbb{R}^{2(2n+1)} ~ | ~ f \left( x-y, s- t + 2ax^t J y\right) = 0, \quad \rho\left( (x,t)\cdot (y,s)^{-1} \right) \sim 1 \}.
\end{eqnarray*}
From Theorem \ref{factor}, we have
\begin{eqnarray*}
f(x,t) = \gamma (|x|^2 - c_1 t ) (|x|^2 + c_1 t) (|x|^2 - c_2 t )(|x|^2 + c_2 t).
\end{eqnarray*}
We need to show that at each point $P_0 \in S$, $\det (Df)$ vanishes of 1 order in each null directions of $d\pi_L$ and $d\pi_R$ at $P_0$.
 Let $P_0 = (x_0, t_0, y_0, s_0)$ and we may assume that $P_0$ is contained in 
\begin{eqnarray*}
S_1 =: \{ (x,t,y,s) \in \mathbb{R}^{2(2n+1)} ~ | ~ |x-y|^2 - c_1 (s - t + 2a x^t J y) = 0 \}.
\end{eqnarray*}
We may identifiy the manifold $C_{\Phi} = \{ \left( (x,t), \Phi_{(x,t)}, (y,s), - \Phi_{(y,s)} \right) \}$ with an open set in $\mathbb{R}^{(2n+1)} \times \mathbb{R}^{(2n+1)}$ by the diffeomorphsim $\phi : \mathbb{R}^{(2n+1)} \times \mathbb{R}^{(2n+1)} \rightarrow S$ given by
\begin{eqnarray*}
\phi (x,t,y,s) = \left( (x,t), \Phi_{(x,t)}, (y,s), - \Phi_{(y,s)} \right).
\end{eqnarray*}
Let the null direction $v_L$ in $\mathbb{R}^{2(2n+1)}$ of $d \pi_L$ at $P_0$. It means that
\begin{eqnarray*}
\begin{pmatrix} I & 0 \\ \frac{\partial^2 \Phi}{\partial_{(x,t)} \partial_{(x,t)} } & \frac{\partial^2 \Phi}{\partial_{(y,s)}\partial_{(x,t)}}  \end{pmatrix} v^t = 0
\end{eqnarray*}
Thus, $v$ is of the form $v = (0,0,z,w)$ with $ w \in \mathbb{R}^{2n}, s \in \mathbb{R}$ and $(w,s)$ satisfies
\begin{eqnarray*}
\frac{\partial^2 \Phi}{\partial_{(y,s)}\partial_{(x,t)}} \begin{pmatrix} z^t \\ w \end{pmatrix} = 0
\end{eqnarray*} 
To verify $\det H(x,t,y,s)$ vanishes of order 1 in the direction $v_L$, it suffice to show that $v_L$ is not orthogonal to the gradient vector $v_g$ of $\det H(x,t,y,s)$ at $P_0$.
From a direct calculation, we get the gradient vector $v_g$ as 
\begin{eqnarray*}
\left. D_{(x,t),(y,s)} \Phi \left( (x,t) \cdot (y,s)^{-1}\right)\right|_p = \left( 2(x-y) - 2 a c_{\beta,1} a J y, ~ - c_{\beta,1}, ~ -2 (x-y) - 2 a c_{\beta,1} x^{t} J,~ c_{\beta,1} \right)
\end{eqnarray*}
Suppose that $v_L$ and $v_g$ are orthogonal. We are going to find a contradiction. 
 It means that 
\begin{eqnarray*}
-2 (x-y) \cdot z - 2 a c_{\beta,1} x^t J \cdot z + c_{\beta,1} w = 0
\end{eqnarray*}
From \eqref{matrix}, we have
\begin{eqnarray*}
(\frac{\partial^2}{\partial x_i \partial y_j} \Phi) = A(y) \left[ (\partial_i \partial_j \Phi) - 2a \partial_t \rho \begin{pmatrix} J & 0 \\ 0 & 0 \end{pmatrix} \right] A(x)^t \cdot \begin{pmatrix} z^t \\ w \end{pmatrix}
\end{eqnarray*}
and
\begin{eqnarray*}
A(x)^t \cdot \begin{pmatrix} z^t \\ w \end{pmatrix} &=&
\begin{pmatrix} 1 & 0 & 0 & \cdots & 0 & 0 \\ 0 & \ddots & 0 & \cdots & 0 &0 \\
0&0&1  & \cdots &0&\vdots \\  0 & 0 &0& \ddots & 1 & 0 \\ 2a x_{n+1} & \cdots  & -2a x_1 & \cdots & -2a x_n & 1 
\end{pmatrix} 
\begin{pmatrix}
z_1 \\ z_2 \\ \vdots \\ z_{2n} \\ w 
\end{pmatrix}
\\
&=& \begin{pmatrix} z_1, & z_2,  & \cdots, & z_{2n}, & 2a(x_{n+1} z_1 + \cdots + x_{2n} z_n - x_1 z_{n+1} - \cdots - x_n z_{2n} ) + w
\end{pmatrix}^t
\end{eqnarray*}
On the other hand, from the orthogonal assumption, we have
\begin{eqnarray*}
z \cdot \left( -2(x-y) - 2 c_{\beta,1} a x^t J \right) + w \cdot c_{\beta,1} = 0
\end{eqnarray*}
It means that
\begin{eqnarray*}
2a (x_{n+1} z_1 + \cdots - x_n z_{2n} ) + w = \frac{2 (x-y) \cdot z } {c_{\beta,1}}
\end{eqnarray*}
so 
\begin{eqnarray*}
A(x)^{t} \cdot \begin{pmatrix} z^t \\ w \end{pmatrix} = \begin{pmatrix} z_1, & z_2, & \cdots, & z_{2n}, & \frac{2(x-y) \cdot z }{c_{\beta,1}} \end{pmatrix}^t
\end{eqnarray*}
Observe that
\begin{eqnarray*}
\left[ (\partial_i \partial_j \Phi) - 2 a\partial_t \rho \begin{pmatrix} J & 0 \\ 0&0 \end{pmatrix} \right] = (\beta+4) \begin{pmatrix} |x|^4 x_1^2 & \cdots & |x|^4 x_1 x_n & |x|^2 x_1 \frac{t}{2} 
\\
\vdots & \ddots & \vdots & \vdots 
\\
|x|^4 x_n x_1 & \cdots & |x|^4 x_n^2 & |x|^2 x_n \frac{t}{2} 
\\
|x|^2 \frac{t}{2} x_1 & \cdots & |x|^2 \frac{t}{2} x_n & \frac{t^2}{4} 
\end{pmatrix} - (|x|^4 + t^2 ) \begin{pmatrix} J & 0 \\ 0 & \frac{1}{2} \end{pmatrix}. 
\end{eqnarray*}
Here, $ x \rightarrow x-y $ and $ t \rightarrow t-s + 2a x^t J y = \frac{|x-y|^2}{c_{\beta,1}}$. 
So, if we calculate with the bottom row, the following should hold
\begin{eqnarray*}
(\beta + 4) \left[ |x-y|^2 \cdot \frac{1}{2} \frac{|x-y|^2}{c_{\beta,1}} (x-y) \cdot z + \frac{|x-y|^4}{c_{\beta,1}^2} \cdot \frac{2}{c_{\beta,1}} (x-y) \cdot z \right] - \frac{1}{2} ( |x-y|^4 + \frac{|x-y|^4}{c_{\beta,1}^2} )\frac{2}{c_{\beta,1}} (x-y) \cdot z = 0
\end{eqnarray*}
Rearranging it, we obtain
\begin{eqnarray*}
\left[ \frac{\beta+2}{2 c_{\beta,1}} + \frac{1}{c_{\beta,1}^3}\right] |x-y|^4~ (x-y) \cdot z = 0. 
\end{eqnarray*}
And 
Thus we have $ (x-y) \cdot z =0$. So 
\begin{eqnarray*}
A(x)^t \cdot \begin{pmatrix} z \\ w \end{pmatrix} = \begin{pmatrix} z_1, z_2, \cdots, z_{2n}, 0\end{pmatrix}^{t}
\end{eqnarray*}
and  $ H_1 (x,t,y,s) \cdot \begin{pmatrix} (z_1,z_2,\cdots, z_{2n})\end{pmatrix}^{t} =0 $. But this is a contradiction because $\det H_1 \neq 0$. 
\

We can also have the same conclusion for  $d\pi_R$ with the same argument. 
\end{proof}

\begin{proof}[proof of Theorem \ref{LA} and Theorem \ref{LB}]
The theorems follow immediately when we use Theorem \ref{degenerate} based on the folding types established in  Theorem \ref{folding}. 
\end{proof}


\begin{thebibliography}{20}

 \bibitem{ho} L. H\"ormander, Oscillatory integrals and multipliers on $FL^p$, Ark. Math. \textbf{11}, (1973), 1-11. 

\bibitem{cu} Cuccagna, Scipio $L^2$L2 estimates for averaging operators along curves with two-sided $k$k-fold singularities. Duke Math. J. 89 (1997), no. 2, 203-216.

\bibitem{Fe} C. Fefferman, Inequalities for strongly singular convolution operators, Acta Math., \textbf{124} (1970), pp. 9-36.

\bibitem{Fe2} C. Fefferman and E. M. Stein, $H^p$ spaces of several variables, Acta Math., \textbf{129} (1972), pp. 137-193.

\bibitem{GR} A. Greenleaf and A. Seeger, Oscillatory integral operators with low-order degeneracies. Duke Math. J. 112 (2002), no. 3, 397-420.

\bibitem{GR2}\bysame,  Oscillatory and Fourier integral operators with degenerate canonical relations. Proceedings of the 6th International Conference on Harmonic Analysis and Partial Differential Equations (El Escorial, 2000). Publ. Mat. 2002, Vol. Extra, 93-141.

\bibitem{hi} I. I. Hirschman, Multiplier Transforms I, Duke Math. J., 26 (1956), pp. 222–242.
\bibitem{laghi} N. Laghi and N Lyall, Strongly singular integral operators associated to different quasi-norms on the Heisenberg group. Math. Res. Lett. \textbf{14} (2007), 825-838.

\bibitem{laghi2} \bysame , Strongly singular Radon transforms on the Heisenberg group and folding singularities Pacific J. Math. \textbf{233} (2007), 403-415.

\bibitem{lyall} N. Lyall, Strongly singular convolution operators on the Heisenberg group. Trans. Amer. Math. Soc. \textbf{359} (2007), 4467-4488  .

\bibitem{PS} Y. Pan and C.D. Sogge, Oscillatory integrals associated to folding canonical relations, Colloq. Math. \textbf{60/61}, (1990), 413-419.

\bibitem{PSam} Y. Pan and G. Sampson, The complete $(L^p,L^p)$ mapping properties for a class of oscillatory integrals. J. Fourier Anal. Appl. \textbf{4} (1998), 93-103.

\bibitem{Sj} P. Sj\"olin, Convolution with oscillating kernels. Indiana Univ. Math. J. 30 (1981), no. 1, 47-55.

\bibitem{St} E. M. Stein, harmonic Analysis: Real-Variable Methods, Orthogonality, and Oscillatory integrals, Princeton Unive. Press, Princeton, 1993.

\bibitem{Wainger} S. Wainger, Special trigonometric series in $k$-dimensions, Memoirs of the AMS \textbf{59}, (1965), American Math. Society.


\end{thebibliography}
\end{document}